\documentclass[12pt]{amsart}
\usepackage[T1]{fontenc}
\usepackage{amssymb}
\let\SavedRightarrow=\Rightarrow
\usepackage{marvosym}
\let\Rightarrow=\SavedRightarrow

\newcommand{\Tee }{\mathcal T}

\newcommand{\cl}{\operatorname{cl}}
\renewcommand{\int}{\operatorname{Int}}

\newtheorem{thm}{Theorem}[section]
\newtheorem{pro}[thm]{Proposition}
\newtheorem{lem}[thm]{Lemma}
\newtheorem{que}[thm]{Question}
\newtheorem{cor}[thm]{Corollary}

\begin{document}

\title{Skeletally generated spaces and absolutes}

\author{V. Valov}
\address{Department of Computer Science and Mathematics, Nipissing University,
100 College Drive, P.O. Box 5002, North Bay, ON, P1B 8L7, Canada}
\email{veskov@nipissingu.ca}
\thanks{The author was supported by NSERC Grant 261914-13}



\begin{abstract}
 Some properties of skelatally generated spaces are established. In particular, it is shown that any
 compactum co-absolute to a $\kappa$-metrizable compactum is skeletally generated. We also prove that
 a compactum $X$ is skeletally generated if and only if its superextension $\lambda X$ is skeletally
 Dugundji and raise some natural questions.
\end{abstract}

\maketitle \markboth{}{Skeletally generated spaces}

\section{Introduction}

In this paper we provide more properties of skeletally generated spaces introduced in \cite{vv}.
It was shown in \cite{vv} that skeletally generated spaces coincide with
I-favorable  spaces \cite{dkz} (for compact spaces that was actually done in \cite{kp8}, see also \cite{kp7}
and \cite{kp9}).
It is interesting that at first view skeletally generated spaces are quite different from I-favorable spaces.
I-favorable spaces were defined as spaces for which the first player has a winning strategy when two players
play the so called open-open game, while skeletally generated spaces can be considered as a skeletal counterpart
of $\kappa$-metrizable compacta \cite{s76} (everywhere in this paper by a compactum we mean a compact Hausdorff space).

Recall that a map $f:X \to Y$ is called
\textit{skeletal} \cite{mr} (resp., {\em semi-open}) if the set $\int_Y\cl_Y f(U)$ (resp., $\int_Yf(U)$) is non-empty, for
any $U\in \Tee_X$. Obviously, every semi-open map is skeletal, and both notions are equivalent for closed maps.

The paper is organized as follows. Some properties of skeletally generated spaces are provided in Section 2.
It is shown that every space co-absolute with a skeletally generated space is also skeletally generated, see Theorem 2.4 (recall that two  spaces
are co-absolute if their absolutes are homeomorphic). In particular, any space co-absolute to a $\kappa$-metrizable compactum is skeletally generated (Corollary 2.5). Section 3 is devoted to the connection of skeletally generated and skeletally Dugundji spaces introduced in \cite{kpv1}.
It is well known that a compactum $X$ is $\kappa$-metrizable if and only if its superextension is a Dugundji space, see \cite{i1} and \cite{s81}.
Theorem 3.1 states that the same connection holds between skeletally generated and skeletally Dugundji compacta (this result, in different terminology, was announced without a proof in \cite{va1}; let us also mention that for zero-dimensional compacta the equivalence $(i)\Rightarrow (iii)$ from Theorem 3.1 was established in \cite[Theorem 5.5.9]{hsh}).

Corollary 2.5, mentioned above, suggests the following problem:
\begin{que}
Is is true that a compactum is skeletally generated if and only if it is co-absolute to a $\kappa$-metrizable compactum?
\end{que}
Obviously, Question 1.1 is interesting for compacta of weight greater that $\aleph_1$. For zero-dimensional compacta this question was posed by 
Heindorf and Shapiro \cite[Section 5.5, p. 140]{hsh}.

Since every skeletally Dugundji space is skeletally generated \cite[Corollary 3.4]{kpv1}, one can ask if there is any skeletally generated space which is not skeletally Dugundji. The referee of this paper pointed out that Heindorf and Shapiro \cite[Proposition 6.3.2]{hsh} established such an example. 
The existence of a skeletally generated space which is not skeletally Dugundji provides a solution to a Question 1.13 from \cite{dkz} whether every compact I-favorable
space is co-absolute to a dyadic space. Indeed, suppose that every I-favorable compactum is co-absolute to a dyadic compactum. Since, by \cite[Corollary 4.5]{kpv1}, dyadic compacta are  skeletally Dugundji, it follows that every compact I-favorable
space is co-absolute to a skeletally Dugundji space. Consequently, all compact I-favorable
spaces would be skeletally Dugundji.


Our results for 0-dimensional compact spaces can be translated using the language of Boolean algebras, see \cite{hsh}. For example, Question 1.1 is   equivalent to the question whether each regularly filtered algebra is co-complete to an rc-filtered one (see \cite{hsh}, p. 140).

All spaces in this paper are Tychonoff and the maps are continuous.

\section{Skeletally generated spaces}

In this section we provide some properties of skeletally generated spaces.

For a given space $X$ we introduce an order on the set of all maps having domain $X$. If $\phi_1$ and $\phi_2$ are two such
maps, we write $\phi_1\prec\phi_2$ if there exists a map
$\phi\colon\phi_1(X)\to\phi_2(X)$ such that $\phi_2=\phi\circ\phi_1$. The notation $\phi_1=\phi_2$ means that the map $\phi$ is
a homeomorphism between $\phi_1(X)$ and $\phi_2(X)$.
We say that $X$ is skeletally generated \cite{vv} if there exists an inverse
system $\displaystyle S=\{X_\alpha, p^{\beta}_\alpha, A\}$ of
separable metric spaces $X_\alpha$ such that:
\begin{itemize}
\item[(1)] All bonding maps $p^{\beta}_\alpha$ are surjective and
skeletal;
\item[(2)] The index set $A$ is $\sigma$-complete (every countable chain in
$A$ has a supremum in $A$);
\item[(3)] For every countable chain $\{\alpha_n:n\geq 1\}\subset A$ with
$\beta=\sup\{\alpha_n:n\geq 1\}$ the space $X_\beta$ is a (dense)
subset of
$\displaystyle\lim_\leftarrow\{X_{\alpha_n},p^{\alpha_{n+1}}_{\alpha_n}\}$;
\item[(4)] $X$ is embedded in $\displaystyle\lim_\leftarrow
S$ such that $p_\alpha(X)=X_\alpha$ for each $\alpha$, where
$p_\alpha\colon\displaystyle\lim_\leftarrow S\to X_\alpha$ is the
$\alpha$-th limit projection;
\item[(5)] For every bounded continuous
real-valued function $f$ on $X$ there exists $\alpha\in A$ such that $p_\alpha\prec f$.
\end{itemize}

Condition $(4)$ implies that $X$ is a dense subset of
$\displaystyle\lim_\leftarrow S$.
An inverse system $S$ with surjective bonding maps satisfying conditions $(2)$ and $(3)$ is called
almost $\sigma$-continuous.  If there exists an almost $\sigma$-continuous system $S$ satisfying
condition $(4)$, we say that $X$ is the almost limit of $S$, notation $X=\mathrm{a}-\displaystyle\lim_\leftarrow S$.

The following characterizations of skeletally generated spaces was established in \cite{vv} (see also \cite{kp8} for
the equivalence of items $(i)$ and $(ii)$ in case $X$ is compact)

\begin{pro}\cite{vv}
For a space $X$ the following are equivalent:
\begin{itemize}
\item[(i)] $X$ is skeletally generated;
\item[(ii)] $X$ is $\mathrm{I}$-favorable;
\item[(iii)] Every $C^*$-embedding of $X$ in another space
is $\pi$-regular.
\end{itemize}
\end{pro}

Here, $X$ is called $\mathrm{I}$-favorable \cite{dkz} if there exists a function $\sigma:\bigcup\{\Tee_X^n:n\geq
0\}\to\Tee_X$, where $\Tee_X$ is the topology of $X$, such that for each sequence
$B_0, B_1,..,$ of non-empty open subsets of $X$ with
$B_0\subset\sigma(\varnothing)$ and  $B_{k+1}\subset\sigma(B_0,B_1,..,B_k)$ for each $k$,
the union  $\bigcup_{k\geq 0}B_k$ is dense in $X$. We say that a subspace $X$ of a space $Y$ is $\pi$-regularly embedded in $Y$ if there exists a function $\mathrm{e}\colon\Tee_X\to\Tee_Y$ between the topologies of $X$ and $Y$ such that:
\begin{itemize}
\item[($\mathrm{e}$1)] $\mathrm{e}(\varnothing)=\varnothing$ and $\mathrm{e}(U)\cap X$ is a dense subset of $U$;
\item[($\mathrm{e}$2)] $e(U)\cap\mathrm{e}(V)=\varnothing$ for any $U,V\in\Tee_X$ provided $U\cap V=\varnothing$.
\end{itemize}
The operator $\mathrm{e}$ is called strongly $\pi$-regular if, in additional, it satisfies condition $(e3)$ below.
\begin{itemize}
\item[($\mathrm{e}$3)] $\mathrm{e}(U\cap V)=e(U)\cap\mathrm{e}(V)$ for any $U,V\in\Tee_X$.
\end{itemize}
Note that $\pi$-regular embeddings were introduced in \cite{va1}, while strongly $\pi$-regular embeddings were considered
in \cite{sh1} under the name $\pi$-regular.

Next lemma follows from Daniels-Kunen-Zhou's observation \cite[Fact 1.3, p.207]{dkz}. For completeness, we provide here a proof of that fact. 
\begin{lem}
Let $f\colon X\to Y$ be a closed irreducible map. Then $X$ is a skeletally generated space if and only if $Y$ is so.
\end{lem}

\begin{proof}
Assume $Y$ is skeletally generated. So, by Proposition 2.1, there exists a function $\sigma_Y:\bigcup\{\Tee_Y^n:n\geq
0\}\to\Tee_Y$ such that for each sequence $B_0, B_1,..,$ of non-empty open subsets of $Y$ with
$B_0\subset\sigma_Y(\varnothing)$ and  $B_{k+1}\subset\sigma_Y(B_0,B_1,..,B_k)$ for each $k$,
$\bigcup_{k\geq 0}B_k$ is dense in $Y$. Since $f$ is irreducible, the set $f^{\sharp}(V)=\{y\in Y: f^{-1}(y)\subset V\}$
is non-empty and open in $Y$ for every open $V\subset X$. Therefore, the function $\sigma_X:\bigcup\{\Tee_X^n:n\geq
0\}\to\Tee_X$, $\sigma_X(\varnothing)=f^{-1}(\sigma_Y(\varnothing))$ and $\sigma_X(V_1,V_2,..,V_n)=f^{-1}\big(\sigma_Y(f^{\sharp}(V_1),f^{\sharp}(V_2),..,f^{\sharp}(V_n))\big)$ is well
defined. If $V_0,V_1,..$ is a sequence of non-empty open subsets of $X$ with $V_0\subset\sigma_X(\varnothing)$ and  $V_{k+1}\subset\sigma_X(V_0,V_1,..,V_k)$ for each $k$, then $U_0\subset\sigma_Y(\varnothing)$ and
$U_{k+1}\subset\sigma_Y(U_0,U_1,..,U_k)$, where $U_k=f^{\sharp}(V_k)$, $k\geq 0$. Hence, $\bigcup_{k\geq 0}U_k$ is
dense in $Y$. Consequently, $\bigcup_{k\geq 0}f^{-1}(U_k)$ is
dense in $X$ (recall that $f$ is closed and irreducible).  Finally, because $f^{-1}(U_k)\subset V_k$, we obtain that
$\bigcup_{k\geq 0}V_k$ is dense in $X$. Thus, $X$ is skeletally generated.

The map $f$ being irreducible and closed is skeletal. Then, according to \cite[Lemma 1]{kp9}, $Y$ is $\mathrm{I}$-favorable
provided so is $X$.
\end{proof}

\begin{cor}
If $X$ is skeletally generated, then so is each compactification of $X$.
\end{cor}

\begin{proof}
It was shown in \cite{vv} that $\beta X$ is skeletally generated. Then Lemma 2.2 completes the proof
because any compactification of $X$ is an irreducible image of $\beta X$.
\end{proof}

\begin{thm}
Any space co-absolute to a skeletally generated space is skeletally generated.
\end{thm}

\begin{proof}
Suppose the spaces $X$ and $Y$ are co-absolute and $Z$ is their common absolute. Then there are closed irreducible maps
$\theta_X\colon Z\to X$ and $\theta_Y\colon Z\to Y$. If $X$ is skeletally generated, then $Z$ is also skeletally generated
(by Lemma 2.2). Applying again Lemma 2.2, we conclude that $Y$ is skeletally generated.
\end{proof}

\begin{cor}
Any space co-absolute to a $\kappa$-metrizable compactum is skeletally generated.
\end{cor}

Recall that $\kappa$-metrizable compacta can be also characterized as the compact spaces $X$ possessing a
lattice in the sense of Shchepin \cite{s76} consisting of open maps. This means that there exists a family $\Psi$ of open maps
with domain $X$ such that:
\begin{itemize}
\item[(L1)] For any map $f\colon X\to f(X)$ there exists $\phi\in\Psi$ with $\phi\prec f$ and $w(\phi(X))\leq w(f(X)$;
\item[(L2)] If $\{\phi_\alpha:\alpha\in A\}\subset\Psi$ such that the diagonal product
$\triangle\{\phi_{\alpha_i}:\alpha_i\in A, i=1,..,k\}$ belongs to $\Psi$ for any finitely many $\alpha_i\in A$, then $\triangle\{\phi_{\alpha}:\alpha\in A\}\in\Psi$.
\end{itemize}

\begin{pro}
Every skeletally generated space has a lattice of skeletal maps.
\end{pro}

\begin{proof}
We consider $X$ as a $C$-embedded subset of $M=\mathbb R^A$ for some $A$. Then there exists a strongly $\pi$-regular operator $\mathrm{e}\colon\Tee_X\to\Tee_M$, see Proposition 2.1. For any countable set $B\subset A$ let $\mathcal B_B$ be an open
base for $\mathbb R^B$ of cardinality $|B|\leq\aleph_0$. e say that a set $B\subset A$ is $\mathrm{e}$-admissible if
$$\pi_B^{-1}(\pi_B(\overline{\mathrm{e}(\phi^{-1}(U)\cap X)}))=\overline{\mathrm{e}(\pi_B^{-1}(U)\cap X)}\hbox{~}\mbox{for all}\hbox{~}U\in\mathcal B_B,$$ where $\pi_B\colon M\to\mathbb R^B$ is the projection.
 The arguments from the proof of \cite[Proposition 3.1(ii)]{kpv1} imply that all maps $\phi_B=\pi_B|X$, where $B$ is $\mathrm{e}$-admissible,
are skeletal and form a lattice for $X$.
\end{proof}

Replacing $\kappa$-metrizable compacta in Corollary 2.5 with spaces possessing a lattice of open maps, we obtain a little bit stronder result.

\begin{pro}
Let $X$ be a space with a lattice of open maps. Then every space co-absolute to $X$ has a lattice of semi-open maps.
\end{pro}

\begin{proof}
The proof follows the arguments from the proof of \cite[Theorem 2.2]{kpv1}. The only difference is that we use Proposition 2.8 below instead of
Proposition 2.1 from \cite{kpv1}.
\end{proof}

\begin{pro}
Let $X$ be $C$-embedded in $\mathbb R^\Gamma$ for some $\Gamma$ and $X$ has a lattice $\Psi$ of quotient maps. Then the family $\mathcal A=\{B\subset\Gamma: \phi_B\in\Psi\hbox{~}\}$, where $\phi_B=\pi_B|X\colon X\to\pi_B(X)$ is the restriction of the projection $\pi_B:\mathbb R^\Gamma\to\mathbb R^B$, has the following properties:
\begin{itemize}
\item[(i)] The union of any increasing subfamily of $\mathcal A$ belongs to $\mathcal A$;
\item[(ii)] Every $A\subset\Gamma$ is contained in some $B\in\mathcal A$
with $|A|=|B|$.
\end{itemize}
\end{pro}

\begin{proof}
Suppose $\{B_\alpha\}$ be an increasing family of subsets of $\Gamma$ with $B_\alpha\in\mathcal A$ for all $\alpha$, and $B=\cup B_\alpha$.
 Then for any finitely many $B_{\alpha_i}$, $i=1,..,n$, there exists $1\leq j\leq n$ such that $B_{\alpha_j}=\bigcup_{i=1}^{i=n}B_{\alpha_i}$. Hence, $\triangle_{i=1}^{i=n}\phi_{B_{\alpha_i}}=\phi_{B_{\alpha_j}}\in\Psi$.  Consequently, by (L2),
 $\phi_B=\triangle\phi_{B_\alpha}\in\Psi$ and $B\in\mathcal A$.

 Assume $C\subset\Gamma$ is an infinite set of cardinality $|C|=\tau$. We construct by
induction an increasing sequence $\{B(k)\}\subset\Gamma$ and a sequence $\{\phi_k\}\subset\Psi$ such that $B(1)=C$, $|B(k)|=\tau$,
$w(\phi_k(X))\leq\tau$ and
$\phi_{B(k+1)}\prec\phi_k\prec\phi_{B(k)}$ for all $k$. Suppose the construction is done up to level $k$ for some $k\geq 1$.  We consider each
$\phi_k(X)$ as a subspace of $\mathbb R^{\tau}$. Since $X$ is $C$-embedded
in $\mathbb R^\Gamma$, there exists a map $g_k\colon\mathbb R^\Gamma\to\mathbb R^{\tau}$ extending $\phi_k$. Then $g_k$ depends on $\tau$
many coordinates of $\mathbb R^\Gamma$, so we can find a set $B(k+1)\subset\Gamma$ of cardinality $\tau$ containing $B(k)$ such that $\pi_{B(k+1)}\prec g_k$.
Consequently, $\phi_{B(k+1)}\prec\phi_k$. Next, by condition $(\mathrm{L1})$, there exists $\phi_{k+1}\in\Psi$ with $\phi_{k+1}\prec\phi_{B(k+1)}$ and
$w(\phi_{k+1}(X))\leq\tau$. This completes the construction. Finally, let $B=\bigcup_{k=1}^{\infty}B(k)$ and
$\phi=\triangle_{k=1}^{\infty}\phi_k$. Obviously, $|B|=\tau$ and $\phi_B=\phi\in\Psi$. Hence, $C\subset B\in\mathcal A$.
\end{proof}

\section{Skeletally Dugundji spaces}

We say that a space $X$ is {\em skeletally Dugundji} \cite{kpv1} if there exists
an inverse system
$\displaystyle S=\{X_\alpha, p^{\beta}_\alpha, \alpha<\beta<\tau\}$ with surjective skeletal bonding maps, where $\tau$ is identified
with the first
ordinal $\omega(\tau)$ of cardinality $\tau$, satisfying the following conditions: (i) $X_0$ is a separable metric space
and all maps $p^{\alpha+1}_\alpha$ have metrizable kernels (i.e., there exists a separable metric space $M_\alpha$ such that
$X_{\alpha+1}$ is embedded in $X_{\alpha}\times M_\alpha$ and $p^{\alpha+1}_\alpha$ coincides with the restriction $\pi|X_{\alpha+1}$ of
the projection $\pi\colon X_{\alpha}\times M_\alpha\to X_{\alpha}$);
(ii) for any limit ordinal $\gamma<\tau$ the space $X_\gamma$ is a (dense)
subset of
$\displaystyle\lim_\leftarrow\{X_{\alpha},p^\beta_\alpha, \alpha<\beta<\gamma\}$;
(iii) $X$ is embedded in $\displaystyle\lim_\leftarrow
S$ such that $p_\alpha(X)=X_\alpha$ for each $\alpha$; (iv) for every bounded continuous
real-valued function $f$ on $\displaystyle\lim_\leftarrow S$ there exists $\alpha\in A$ such that $p_\alpha\prec f$.
It was shown in \cite{kpv1} that $X$ is skeletally Dugundji provided every $C^*$-embedding of $X$ in another space is strongly $\pi$-regular.

There exists a tight connection between openly generated compacta and their superextensions.
Ivanov \cite{i1} proved that if $X$ is openly generated compactum, then its superextension $\lambda X$ is a Dugundji space (the other implication is also true, see \cite{s81}).  
Theorem 3.1 below provides a similar connection between skeletally generated and skeletally Dugundji compacta (let us explicitly mention that for zero-dimensional compacta the equivalence $(i)\Rightarrow (iii)$ was established in \cite[Theorem 5.5.9]{hsh}).

\begin{thm}
For a compact space $X$ the following are equivalent:
\begin{itemize}
\item[(i)] $X$ is skeletally generated;
\item[(ii)] Every embedding of $\lambda X$ in another space is strongly $\pi$-regular, in particular $\lambda X$ is skeletally Dugundji;
\item[(iii)] $\lambda X$ is skeletally generated.
\end{itemize}
\end{thm}

\begin{proof}
$(i)\Rightarrow (ii)$.
The superextension $\lambda X$ is the set of all maximal linked systems $\xi$ of closed subsets of $X$ (recall that $\xi$ is linked if any two elements of $\xi$ intersect). For any set $H\subset X$ let
$H^+$ be the set of all $\xi\in\lambda X$ such that $F\subset H$ for some $F\in\xi$. Then the family $\mathcal B^+$ of all sets of the form
$[U_1^+,U_2^+,..,U_k^+]=\bigcap_{i=1}^{i=k}U_i^+$, where $U_1,..,U_k$ are open in $X$, is a base for the topology of $\lambda X$.
We consider $\lambda X$ as a subset of $\mathbb I^\tau$ for some cardinal $\tau$. Since, by \cite{bmk}, $\lambda X$ is also skeletally generated,   according to Proposition 2.1, there exists a $\pi$-regular operator $\mathrm{e}\colon\Tee_{\lambda X}\to\Tee_{\mathbb I^\tau}$. We define a set-valued map
$r\colon\mathbb I^\tau\to\lambda X$ by $$r(y)=\bigcap\{(\overline{U_i})^+:y\in\mathrm{e}([U_1^+,U_2^+,..,U_k^+])\}\hbox{~}\mbox{if}\hbox{~}
y\in\bigcup\{\mathrm{e}(W):W\in\mathcal B^+\}$$ and $r(y)=\lambda X$ otherwise. This definition is correct because for every $y\in\mathbb I^\tau$ the system $\gamma_y=\{W\in\mathcal B^+:y\in\mathrm{e}(W)\}$ is linked, so $r(y)\neq\varnothing$. It is easily seen that $r$ is upper semi-continuous. Then for any open $W\subset\lambda X$ the set $$\mathrm{e}_1(W)=\bigcup\{r^\sharp([U_1^+,U_2^+,..,U_k^+]):\bigcap_{i=1}^{i=k}(\overline{U_i})^+\subset W\}$$ is open in $\mathbb I^\tau$, where $r^\sharp([U_1^+,U_2^+,..,U_k^+])=\{y:r(y)\subset [U_1^+,U_2^+,..,U_k^+]\}$. It follows directly from our definition that $\mathrm{e}_1$ satisfies conditions $(\mathrm{e}2)$ and $(\mathrm{e}3)$ from the definition of strongly $\pi$-regular operator. We are going to show that $\mathrm{e}_1(W)\cap\lambda X$ is dense in $W$ for all open $W\subset\lambda X$.

Suppose $W\subset\lambda X$ is open and $\xi\in W$. Let $\xi\in\tilde{G}\subset W$, where $\tilde{G}=[G_1^+,G_2^+,..,G_k^+]\in\mathcal B^+$.
Then for each $i\leq k$ there exists $F_i\in\xi$ and open sets $V_i$ and $U_i$ in $X$ such that $F_i\subset V_i\subset\overline{V_i}\subset U_i\subset\overline{U_i}\subset G_i$. Take $\eta\in\mathrm{e}([V_1^+,..,V_k^+])\cap\lambda X$. Then  $r(\eta)\subset\bigcap_{i=1}^{i=k}(\overline{V_i})^+\subset [U_1^+,..,U_k^+]$. Hence, $\eta\in r^\sharp([U_1^+,U_2^+,..,U_k^+])\cap\lambda X$. Since
$\bigcap_{i=1}^{i=k}(\overline{U_i})^+\subset W$, we have $\eta\in\mathrm{e}_1(W)\cap\tilde{G}$. Therefore, $\mathrm{e}_1(W)\cap\lambda X$ is dense in $W$.

$(ii)\Rightarrow (iii)$. This implication is obvious because every skeletally Dugundji space is skeletally generated \cite{kpv1}.

$(iii)\Rightarrow (i)$. Consider $\lambda X$ as a subspace of some $\mathbb I^\Gamma$. According to Proposition 2.6 and Proposition 2.8,
there exists a family $\mathcal A_c$ of countable sets $A\subset\Gamma$ such that: (i) the union of any increasing sequence from $\mathcal A_c$ belongs to $\mathcal A_c$;
(ii) any countable subset of $\Gamma$ is contained in some $A\in\mathcal A_c$; (iii) any map $p_A=\pi_A|(\lambda X)\to\pi_A(\lambda X)$, $A\in\mathcal A_c$, is skeletal, where
$\pi_A\colon\mathbb I^\Gamma\to\mathbb I^A$ denotes the projection. Let $\varphi_A$ be the restriction of $p_A$ on $X$ and $X_A=\varphi_A(X)$
for each $A\in\mathcal A_c$ (we consider $X$ as a naturally embedded subset of $\lambda X$). Then $\lambda\varphi_A$ is a map from $\lambda X$ into $\lambda X_A$.  Denote by $\mathcal B$ the family of all $B\in\mathcal A_c$ such that $p_B=\lambda\varphi_B$. Since the functor $\lambda$ is continuous, it follows from
Shchepin's spectral theorem \cite{s76}, that the union of any increasing sequence from $\mathcal B$ is again in $\mathcal B$,  and any $A\in\mathcal A_c$ is contained in some $B\in\mathcal B$. Therefore, the inverse system $\{X_B, \varphi_B^C, B, C\in\mathcal B\}$ is continuous and its limit is $X$ (here $\varphi_B^C$ is the projection from $X_C$ to $X_B$ provided $B\subset C$). So, it remains to show that each $\varphi _B$, $B\in\mathcal B$, is skeletal.

To this end, let $U\subset X$ be open and $B\in\mathcal B$. Then $U^+$ is open in $\lambda X$ and, since $p_B$ is a closed and skeletal map, the set
$p_B(U^+)$ has a non-empty interior in $\lambda X_B$. So, there are open subsets $V_i$, $i=1,..,k$, of $X_B$ with
$[V_1^+,V_2^+,..,V_k^+]\subset p_B(U^+)$. We claim that $V_i\cap V_j\subset\varphi_B(U)$ for some $i\neq j$. Indeed, otherwise for every $i\neq j$ we can choose
$y_{i,j}=y_{j,i}\in V_i\cap V_j\backslash\varphi_B(U)$. Then $\{F_1=\{y_{1,2},..,y_{1,k}\},..,F_k=\{y_{k,1},..,y_{k,k-1}\}\}$ is a linked system and generates $\eta\in\lambda X_B$ such that $\eta\in\bigcap_{i=1}^{i=k}F_i^+$. Obviously, $\eta\in\bigcap_{i=1}^{i=k}V_i^+$, which contradicts the fact that
$\bigcap_{i=1}^{i=k}V_i^+\subset p_B(U^+)$.
\end{proof}

\textbf{Acknowledgments.} The author would like to express his
gratitude to the referee of this paper whose valuable remarks and suggestions improved the paper.


\end{document}